\documentclass[12pt,a4paper]{amsart}

\usepackage{geometry}                
\usepackage{graphicx}
\usepackage{amssymb}
\usepackage{epstopdf}
\usepackage{hyperref}
\usepackage[usenames]{color}

\input xy
\xyoption{all}
\usepackage{tikz}
\usetikzlibrary{arrows,shapes}

\renewcommand{\epsilon}{\varepsilon}
\renewcommand{\setminus}{\smallsetminus}
\renewcommand{\emptyset}{\varnothing}

\newtheorem{theorem}{Theorem}[section]
\newtheorem{proposition}[theorem]{Proposition}
\newtheorem{corollary}[theorem]{Corollary}
\newtheorem{lemma}[theorem]{Lemma}

\theoremstyle{definition}
\newtheorem{example}[theorem]{Example}

\newtheorem{definition}[theorem]{Definition}
\newtheorem{notation}[theorem]{Notation}
\newtheorem{question}[theorem]{Question}
\newtheorem{remark}[theorem]{Remark}

\newcommand{\Q}{\mathbb Q}
\newcommand{\Z}{\mathbb Z}
\newcommand{\N}{\mathbb N}

\renewcommand{\L}{\mathcal L}

\newcommand{\F}{\operatorname{F}}

\newcommand{\FFF}{\operatorname{F}}

\newcommand{\FP}{\operatorname{FP}}

\newcommand{\cohom}[3]{H^{{\raise1pt\hbox{$\scriptstyle#1$}}}(#2\>\!,#3)}
\newcommand{\tatecohom}[3]%
  {\widehat H^{{\raise1pt\hbox{$\scriptstyle#1$}}}(#2\>\!,#3)}

\newcommand{\Cohom}[3]%
  {H^{{\raise1pt\hbox{$\scriptstyle#1$}}}\big(#2\>\!,#3\big)}
\newcommand{\Tatecohom}[3]%
  {\widehat H^{{\raise1pt\hbox{$\scriptstyle#1$}}}\big(#2\>\!,#3\big)}

\newcommand{\homol}[3]{H_{{\lower1pt\hbox{$\scriptstyle#1$}}}(#2\>\!,#3)}
\newcommand{\homolog}[2]{H_{{\lower1pt\hbox{$\scriptstyle#1$}}}(#2)}

\newcommand{\colim}{\varinjlim}

\renewcommand{\ker}{\operatorname{Ker}}

\newcommand{\mono}{\rightarrowtail}
\newcommand{\epi}{\twoheadrightarrow}

\newcommand{\downlink}{\operatorname{lk}\!\!\downarrow\!\!}

\DeclareMathOperator{\glb}{glb}

\title[Finiteness conditions in generalisations of Thompson's group $V$]{Cohomological finiteness conditions and centralisers in generalisations of Thompson's group $V$.}

\author{C.~Mart\'inez-P\'erez}

\address{Conchita Mart\'inez-P\'erez, Departamento de Matem\'aticas, Universidad de Zaragoza,
50009 Zaragoza, Spain} \email{conmar@unizar.es}

\author{F. Matucci}

\address{Francesco Matucci, D\'epartement de Math\'ematiques, Facult\'e des Sciences d'Orsay,
Universit\'e Paris-Sud 11, B\^atiment 425, Orsay, France}
\email{francesco.matucci@math.u-psud.fr}

\author{B.~ E.~A.~Nucinkis}

\address{Brita E.~A.~Nucinkis, Department of Mathematics, Royal Holloway, University of London, Egham, TW20 0EX, UK.}\email{brita.nucinkis@rhul.ac.uk}

\date{\today} 

\thanks{This work was partially funded by an LMS Scheme 4 grant 41209. The first named author was supported by  Gobierno de Arag\'on, European Regional Development Funds and
MTM2010-19938-C03-03. 
The second author gratefully acknowledges the Fondation Math\'ematique Jacques Hadamard (ANR­-10-CAMP­-0151-­02 - FMJH - Investissement d'Avenir) for the support received during the development of this work. }

\begin{document}

\begin{abstract} We consider generalisations of Thompson's group $V$, denoted $V_r(\Sigma)$, which also include the groups of Higman, Stein and Brin.  We show that, under some mild hypotheses, $V_r(\Sigma)$ is the full automorphism group of a Cantor-algebra. Under some further minor restrictions, we prove
that these groups are of type $\F_\infty$ and that this implies that also centralisers of finite subgroups are of type $\F_\infty$. 
\end{abstract}

\maketitle

\section{Introduction}

\noindent Thompson's group $V$ is defined as a homeomorphism group of the Cantor-set.
The group $V$ has many interesting generalisations such as the Higman-Thompson groups $V_{n,r}$,  \cite{higman}, Stein's generalisations \cite{stein} and Brin's higher dimensional Thompson groups $sV$ \cite{brin1}. All these groups contain any finite group, contain free abelian groups of infinite rank, are finitely presented and of type $\FP_\infty$ (see work by several authors in
\cite{brown2, fluch++, hennigmatucci, desiconbrita2, stein}). 
The first and third authors together with Kochloukova \cite{desiconbrita2, britaconcha} 
further generalise these groups, denoted by $V_r(\Sigma)$ or $G_r(\Sigma)$, as automorphism groups of certain Cantor-algebras. We shall use the notation $V_r(\Sigma)$ in this paper. 
We show in Theorem \ref{fullauto} that they are the full automorphism groups of these algebras.

Fluch, Schwandt, Witzel and Zaremsky \cite{fluch++} 
use Morse-theoretic methods to prove that Brin's groups $sV$ are of type $\F_\infty.$  By adapting their methods we show, Theorem \ref{FPinfty}, that under some restrictions on the Cantor-algebra, 
which still comprehend all families mentioned above, $V_r(\Sigma)$ is of type $\F_\infty.$  We also give some constructions of further examples.

Bleak \emph{et al.} \cite{matuccietc} and the first and the third
authors \cite{britaconcha} show independently that centralisers of finite subgroups $Q$ in $V_{n,r}$ and $V_r(\Sigma)$ can be described as extensions
$$K \mono C_{V_r(\Sigma)}(Q) \epi V_{r_1}(\Sigma) \times \ldots 
\times V_{r_t}(\Sigma),$$
where $K$ is locally finite and $r_1,...,r_t$ are integers uniquely determined by $Q$. It was conjectured in \cite{britaconcha} that these centralisers are of type $\F_\infty$ if the groups $V_r(\Sigma)$ are. In Section \ref{centralisersection} we expand the description of the centralisers given in \cite{matuccietc, britaconcha}, which allows us to
prove that the conjecture holds true. This also implies that any of the generalised $V_r(\Sigma)$ which are of type $\F_\infty$ admit a classifying space for proper actions, which is a mapping telescope of cocompact classifying spaces for smaller families of finite subgroups. In other words, these groups are of Bredon type quasi-$\F_\infty$. For definitions and background the reader is referred to \cite{britaconcha}.

We conclude with giving a description of normalisers of finite subgroups in Section \ref{normaliser}. These turn up in computations of the source of the rationalised Farrell-Jones assembly map, where one needs to compute not only centralisers, but also the Weyl-groups $W_G(Q)=N_G(Q)/C_G(Q).$ For more 
detail see \cite{lueck-reich05}, or \cite{geoghegan-varisco} for an example where these are computed for Thompson's group $T$.

\subsection*{Acknowledgments} We would like to thank Dessislava Kochloukova for helpful discussions regarding Section \ref{fpinfty}, 
Claas R\"over for getting us to think about Theorem \ref{fullauto}
and the anonymous referee for very carefully reading an earlier version of this paper. We are also indebted to this referee for pointing out Lemma \ref{auxmainfinfty} to us.

\section{Background on generalised Thompson groups}

\subsection{Cantor-algebras}

\noindent We shall follow the notation of \cite[Section 2]{britaconcha} and begin by defining the Cantor algebras $U_r(\Sigma)$. Consider a finite set of colours $S=\{1,\ldots,s\}$ and associate to each $i \in S$ an integer  $n_i>1$, called arity of the colour $i.$ Let $U$ be a set on which, for all $i\in S$, the following operations are defined: an $n_i$-ary operation $\lambda_i\, :\,U^{n_i}\to U,$ and $n_i$ 1-ary operations $\alpha^1_i,\ldots,\alpha^{n_i}_i\, ; \alpha^j_i:U\to U.$ Denote $\Omega = \{ \lambda_i, \alpha_i^j \}_{i,j}$ and call $U$ an $\Omega$-algebra. For detail see \cite{Cohn} and \cite{desiconbrita2}. We write these operations on the right. We also consider, for each $i\in S$ and $v\in U,$ the map $\alpha_i:U\to U^{n_i}$ given by
$v\alpha_i:=(v\alpha^1_i,v\alpha^2_i,\ldots,v\alpha^{n_i}_i).$ The maps $\alpha_i$ are called descending operations, or expansions, and the maps $\lambda_i$ are called ascending operations, or contractions. Any word in the descending operations is called a descending word.

A {\sl morphism} between $\Omega$-algebras is a map commuting with all operations
in $\Omega$. Let $\mathfrak{B}_0$ be the category of all $\Omega$-algebras for some $\Omega$. 
An object $U_0(X) \in \mathfrak{B}_0$ 
is a {\sl free object} in $\mathfrak{B}_0$ with $X$ as a {\sl free basis}, if for any $S \in \mathfrak{B}_0$ 
any mapping $\theta : X \to S$ can be extended in a unique way to a morphism 
$U_0(X) \to S$.

For every set $X$ there is an $\Omega$-algebra, free on $X$,  called the {\sl $\Omega$-word algebra on $X$} and denoted by $W_\Omega(X)$ (see \cite[Definition 2.1]{desiconbrita2}).
Let $B\subset W_{\Omega}(X)$, $b\in B$ and $i$ a colour of arity $n_i$. The set
$$(B\setminus\{b\})\cup\{b\alpha_i^1,\ldots,b\alpha_i^{n_i}\}$$
is called a simple expansion of $B$. Analogously, if $b_1,\ldots,b_{n_i}\subseteq B$ are pairwise distinct,
$$(B\setminus\{b_1,\ldots,b_{n_i}\})\cup\{(b_1,\ldots,b_{n_i})\lambda_i\}$$
is a simple contraction of $B$.
A chain of simple expansions (contractions) is an expansion (contraction). A  subset $A\subseteq W_\Omega(X)$ is called {\sl admissible} if it can be obtained from the set $X$ by finitely many expansions or contractions.

\noindent We shall now define the notion of a Cantor-algebra. Fix a finite set $X$ and consider the variety of $\Omega$-algebras satisfying a certain set of identities as follows:

\begin{definition}\label{sigmadef}\cite[Section 2]{britaconcha} We denote by $\Sigma = \Sigma_1 \cup \Sigma_2$  the following set of
  laws in the alphabet $X$.

\begin{itemize}

\item[i)] A set of laws $\Sigma_1$ given by
  $$u\alpha_i\lambda_i=u,$$
 $$(u_1,\ldots,u_{n_i})\lambda_i\alpha_i=(u_1,\ldots,u_{n_i}),$$
for every $u\in W_\Omega(X)$,  $i\in S$, and $n_i$-tuple: $(u_1,\ldots,u_{n_i})\in W_\Omega(X)^{n_i}.$

\item[ii)] A second set of laws
$$\Sigma_2=\bigcup_{1\leq i<i'\leq s}\Sigma_2^{i,i'}$$
 where each $\Sigma_2^{i,i'}$ is either empty or  consists  of the following laws: consider first $i$ and fix a map $f:\{1,\ldots,n_{i}\}\to\{1,\ldots,s\}$. For each $1\leq j\leq n_{i}$, we see  $\alpha_{i}^j\alpha_{f(j)}$ as a set of length 2 sequences of descending operations and
let $\Lambda_{i}=\cup_{j=1}^{n_{i}}\alpha_{i}^j\alpha_{f(j)}$.  Do the same for $i'$ (with a corresponding map $f'$) to get $\Lambda_{i'}$.
We need to assume that $f,f'$ are chosen so that $|\Lambda_i|=|\Lambda_{i'}|$  and  fix a bijection $\phi:\Lambda_{i}\to\Lambda_{i'}$.
Then $\Sigma_2^{i,i'}$ is the set of laws
$$u\nu=u\phi(\nu)\quad \nu\in \Lambda_{i},u\in W_\Omega(X).$$

\end{itemize}

\noindent Factor out of $W_\Omega(X)$ the fully invariant congruence $\mathfrak{q}$ generated by $\Sigma$ to obtain an $\Omega$-algebra   $W_\Omega(X)/\mathfrak{q}$ satisfying the identities in $\Sigma$.

\noindent The algebra $W_\Omega(X)/\mathfrak{q}=U_r(\Sigma),$ where $r=|X|$, is called a {\sl Cantor-Algebra.}
\end{definition}

\noindent
As in \cite{desiconbrita2} we say that $\Sigma$ is {\sl valid}  if for any admissible $Y\subseteq W_\Omega(X)$, we have $|Y|=|\overline Y|$, where $\overline Y$ is the image of $Y$ under the epimorphism $W_\Omega(X) \twoheadrightarrow U_r(\Sigma).$ In particular this implies that $U_r(\Sigma)$ is a free object on $X$ in the class of those $\Omega$-algebras which satisfy the identities $\Sigma$ above.
In other words, this implies that $X$ is a basis.
\noindent If the set $\Sigma$ used to define $U_r(\Sigma)$ is valid, we also say that $U_r(\Sigma)$ is valid.
As done for $W_\Omega(X)$, we say that a subset $A\subset U_r(\Sigma)$ is {\sl admissible} if it can be obtained by a finite number of expansions or contractions from $\overline X$, where expansions and contractions mean the same as before. We shall, from now on,  not distinguish between $X$ and $\overline X$. If $A$ can be obtained 
from a subset $B$ by expansions only, we will say that $A$ is an expansion or a descendant of $B$ and we will write $B\leq A$. If $A$ can be obtained from $B$ by applying a single descending operation, i.e., if
$$A=(B\setminus\{b\})\cup\{b\alpha_i^1,\ldots,b\alpha_i^{n_i}\}$$
for some colour $i$ of arity $n_i$, then we will say that $A$ is a simple expansion of $B$.

\begin{remark}\label{adm-basis}
Every admissible subset is a basis of $U_r(\Sigma)$ (\cite[Lemma 2.5]{desiconbrita2}). Moreover, any set obtained from a basis by performing expansions or contractions is also a basis.
Furthermore, the cardinality $m$ of every admissible subset satisfies $m\equiv r$ mod $d$ for $d:=\text{gcd}\{n_i-1\mid i=1,\ldots,s\}.$
In particular, any basis with $m$ elements can be transformed into one of $r$ elements. Hence $U_r(\Sigma)=U_m(\Sigma)$ and we may assume that $r \leq d.$ 
\end{remark}

The converse of the first statement in Remark \ref{adm-basis} is also true:

\begin{theorem}\label{fullauto} Let $U_r(\Sigma)$ be a valid and bounded Cantor algebra. Then $V_r(\Sigma)$ is the full group of $\Omega$-algebra automorphisms of $U_r(\Sigma)$.
\end{theorem}

\begin{proof}
We need to show that, under our hypotheses, a subset of $U_r(\Sigma)$ 
is admissible if and only if it is a basis. In light of Remark \ref{adm-basis} we only need to show that any basis of $U_r(\Sigma)$ is admissible. Let $Y=\{y_1,...,y_n\}$ be an arbitrary basis. Since $X$ is a basis, it generates all of $U_r(\Sigma)$. Hence, for each $y_i \in Y$ there exists some admissible subset $T_i$ of $U_r(\Sigma)$ containing $y_i$. Now let $Z$ be a common upper bound of the $T_i$, $i=1,...n.$  This exists by \cite[Lemma 2.8]{britaconcha}, using the argument of \cite[Proposition 3.4]{desiconbrita2}. The set $Z$ is an admissible subset containing a set $\widehat Y$ whose elements are obtained by performing   finitely many descending  operations in $Y$. Denote by $\widehat Y_i$ the subsets of $\widehat Y$ given by the following: $\{y_i\} \leq \widehat Y_i$ and $\widehat Y = \cup\widehat Y_i$.
Since $Y$ and $Z$ are bases and $Y \le Z$,
then Remark \ref{tec} below implies that
$\widehat Y_i \cap \widehat Y_j = \emptyset$, for $i \ne j$.
By Remark \ref{adm-basis}, since 
$\widehat Y$ is admissible, it is a basis.  Remark \ref{adm-basis} also implies that $Z$ is a basis. It follows from the definition of free basis, see for example \cite[Page 3]{desiconbrita2}, that no proper subset of a basis is a basis. Hence $\widehat Y=Z$ is admissible, thus $Y$ is as well.
\end{proof}

\begin{remark}\label{tec} Let $B$ be a basis in a valid $U_r(\Sigma)$, and let $A \leq B$.  The fact that $A$ is also a basis implies that for any element $b\in B$ there is a single $A(b)\in A$ such that $A(b)w=b$ for some descending word $w$. In this case we say that $A(b)$ is a prefix of $b$.
\end{remark}

\noindent Any element of $U_r(\Sigma)$ which is obtained from the elements in $X$ by applying descending operations only is called a {\sl leaf}. We denote by $\L$ the set of leaves. Observe that $\L$ depends on $X$.
 Note also that for any leaf $l$ there is some basis $A \geq X$ with $l\in A$. Let $l\in\L$, we put:
 $$l(\L):=\{b\in\L\mid lw=bw'\text{ for descending words $w,w'$}\}$$
and for a set of leaves $B\subseteq\L$ we also put
$$B(\L)=\bigcup_{b\in B}b(\L).$$

We say  that $U_r(\Sigma)$ is {\sl bounded} (see \cite[Definition 2.7]{britaconcha})  if for all admissible subsets $Y$ and $Z$ such that there is some admissible $A\leq Y,Z$, there is a unique least upper bound of $Y$ and $Z$. By a unique least upper bound we mean an admissible subset $T$ such that $Y \leq T$ and $Z \leq T$, and whenever there is an admissible set $S$ also satisfying $Y\leq S$ and $Z\leq S$, then $T \leq S.$

\begin{definition}\label{groups}\cite[Definition 2.12]{britaconcha} Let $U_r(\Sigma)$ be a valid Cantor algebra. $V_r(\Sigma)$ denotes the group of all $\Omega$-algebra automorphisms of $U_r(\Sigma)$, which are induced by a  map $V\to W$, where $V$ and $W$ are admissible subsets of the same cardinality.
\end{definition}

\noindent Throughout we shall denote group actions on the left.

\begin{remark}\label{tecrem1}  For any basis $A \geq X$ and any $g\in V_r(\Sigma)$, there is some $B$ with $A\leq B,gB$. To see it, take $B$ such that $A,g^{-1}A\leq B,$ which exists by \cite[Lemma 2.8]{britaconcha}.
\end{remark}

\medskip\noindent

\subsection{Brin-like groups}\label{newgroups}

\noindent In this section we give some examples of the groups $V_r(\Sigma)$, which generalise both Brin's groups $sV$ \cite{brin1} and Stein's groups $V(l,A,P)$ \cite{stein}.  Furthermore, these groups satisfy the conditions of Definition \ref{complete} below, and we show in Section \ref{fpinfty} that they 
are of type $\F_\infty.$

\begin{example}\label{brinexamples} 
\begin{enumerate}
\item We begin by recalling the definition of the Brin-algebra \cite[Section 2]{desiconbrita2} and \cite[Example 2.4]{britaconcha}:  Consider the set of $s$ colours $S=\{1,\ldots,s\}$, all of which have arity 2, together with the relations:
$\Sigma:=\Sigma_1\cup\Sigma_2$
with $$\Sigma_2:=\{\alpha_i^l\alpha_j^t=\alpha_j^t\alpha_i^l\mid
1\leq i \not= j\leq s; l,t=1,2 \}.$$

\noindent Then $V_r(\Sigma)=sV$ is Brin's group.

\item Furthermore one can also consider  $s$ colours,  all of arity $n_i=n\in \N,$ for all $1\leq i\leq s.$ Let
$$\Sigma_2:=\{\alpha_i^l\alpha_j^t=\alpha_j^t\alpha_i^l\mid
1\leq i \not= j\leq s; 1\leq l,t\leq n \}.$$

\noindent Here $V_r(\Sigma)=sV_n$ is Brin's group of arity $n.$

\noindent It was shown in \cite[Example 2.9]{britaconcha} that in this case $U_r(\Sigma)$ is valid and bounded.

\item We can also mix arities. Consider $s$ colours, each of arity $n_i\in \N$ $(i=1,...,s)$, together with $\Sigma:=\Sigma_1\cup\Sigma_2$
where $$\Sigma_2:=\{\alpha_i^l\alpha_j^t=\alpha_j^t\alpha_i^l\mid
1\leq i \not= j\leq s; 1\leq l\leq n_i; 1\leq t\leq n_j \}.$$

\noindent We denote these mixed-arity Brin-groups by $V_r(\Sigma)=V_{\{n_1\},...,\{n_s\}}.$ 

\noindent The same argument as in \cite[Lemma 3.2]{desiconbrita2} yields that the Cantor-algebra $U_r(\Sigma)$ in this case is also valid and bounded.

\end{enumerate}
\end{example}

\medskip
\begin{tikzpicture}[scale=0.4]

  \draw[black, dashed]
    (1,0) -- (3, 3) -- (5,0);

  \draw[black] (0,-3)--(1,0)--(2,-3);
  \draw[black] (1,-3)--(1,0);
  \draw[black] (4,-3)--(5,0)--(6,-3);
  \draw[black] (5,-3)--(5,0);

      \filldraw (0,-3) circle (0.3pt) node[below=4pt]{$1$};
      \filldraw (1,-3) circle (0.3pt) node[below=4pt]{$2$};
      \filldraw (2,-3) circle (0.3pt) node[below=4pt]{$3$};
      \filldraw (4,-3) circle (0.3pt) node[below=4pt]{$4$};
      \filldraw (5,-3) circle (0.3pt) node[below=4pt]{$5$};
      \filldraw (6,-3) circle (0.3pt) node[below=4pt]{$6$};

  \draw[black]  (11,0) -- (13, 3) -- (15,0);
  \draw[black] (13,0)--(13,3);

  \draw[black, dashed] (10.5,-3)--(11,0)--(11.5,-3);
 \draw[black, dashed] (12.5,-3)--(13,0)--(13.5,-3);
  \draw[black,dashed] (14.5,-3)--(15,0)--(15.5,-3);

      \filldraw (10.5,-3) circle (0.3pt) node[below=4pt]{$1$};
      \filldraw (11.5,-3) circle (0.3pt) node[below=4pt]{$4$};
      \filldraw (12.5,-3) circle (0.3pt) node[below=4pt]{$2$};
      \filldraw (13.5,-3) circle (0.3pt) node[below=4pt]{$5$};
      \filldraw (14.5,-3) circle (0.3pt) node[below=4pt]{$3$};
      \filldraw (15.5,-3) circle (0.3pt) node[below=4pt]{$6$};

\end{tikzpicture}

\bigskip\centerline{Figure 1: Visualising the identities in $\Sigma_2$ for $V_{\{2\},\{3\}}.$}

\begin{example}\label{stein}
We now recall the laws $\Sigma_2$ for Stein's groups \cite{stein}:  Let $P\subseteq\Q_{>0}$ be a finitely generated multiplicative group. Consider a basis of $P$ of the form $\{n_1,\ldots,n_s\}$ with all $n_i \geq 0$ ($i=1,...,s$).
Consider $s$ colours of arities $\{n_1,\ldots,n_s\}$ and let $\Sigma=\Sigma_1\cup\Sigma_2$ with $\Sigma_2$ the set of identities given by the following order preserving identification:
$$\{\alpha_i^1\alpha_j^1,\ldots,\alpha_i^1\alpha_j^{n_j},\alpha_i^2\alpha_j^1,\ldots,\alpha_i^2\alpha_j^{n_j},\ldots,\alpha_i^{n_i}\alpha_j^1,\ldots,\alpha_i^{n_i}\alpha_j^{n_j}\}=$$
$$\{\alpha_j^1\alpha_i^1,\ldots,\alpha_j^1\alpha_i^{n_i},\alpha_j^2\alpha_i^1,\ldots,\alpha_j^2\alpha_i^{n_i},\ldots,\alpha_j^{n_j}\alpha_i^1,\ldots,\alpha_j^{n_j}\alpha_i^{n_i}\},$$
where $i \neq j$ and $i,j\in \{1,...,s\}.$

The resulting Brown-Stein algebra $U_r(\Sigma)$ is valid and bounded, see, for example \cite[Lemma 2.11]{britaconcha}. We denote the resulting groups $V_r(\Sigma)=V_{\{n_1,...,n_s\}}.$

\end{example}

\medskip
\begin{tikzpicture}[scale=0.4]

  \draw[black,dashed]
    (1,0) -- (3, 3) -- (5,0);

  \draw[black] (0,-3)--(1,0)--(2,-3);
  \draw[black] (1,-3)--(1,0);
  \draw[black] (4,-3)--(5,0)--(6,-3);
  \draw[black] (5,-3)--(5,0);

      \filldraw (0,-3) circle (0.3pt) node[below=4pt]{$1$};
      \filldraw (1,-3) circle (0.3pt) node[below=4pt]{$2$};
      \filldraw (2,-3) circle (0.3pt) node[below=4pt]{$3$};
      \filldraw (4,-3) circle (0.3pt) node[below=4pt]{$4$};
      \filldraw (5,-3) circle (0.3pt) node[below=4pt]{$5$};
      \filldraw (6,-3) circle (0.3pt) node[below=4pt]{$6$};

  \draw[black]  (11,0) -- (13, 3) -- (15,0);
  \draw[black] (13,0)--(13,3);

  \draw[black, dashed] (10.5,-3)--(11,0)--(11.5,-3);
 \draw[black, dashed] (12.5,-3)--(13,0)--(13.5,-3);
  \draw[black,dashed] (14.5,-3)--(15,0)--(15.5,-3);

      \filldraw (10.5,-3) circle (0.3pt) node[below=4pt]{$1$};
      \filldraw (11.5,-3) circle (0.3pt) node[below=4pt]{$2$};
      \filldraw (12.5,-3) circle (0.3pt) node[below=4pt]{$3$};
      \filldraw (13.5,-3) circle (0.3pt) node[below=4pt]{$4$};
      \filldraw (14.5,-3) circle (0.3pt) node[below=4pt]{$5$};
      \filldraw (15.5,-3) circle (0.3pt) node[below=4pt]{$6$};

\end{tikzpicture}

\bigskip\centerline{Figure 2: Visualising the identities in $\Sigma_2$ for $V_{\{2,3\}}.$}

\begin{definition}\label{def-brinlike}

Let $S$ be a set  of $s$ colours together with arities $n_i$ for each $i=1,...,s.$ Suppose $S$ can be partitioned into $m$ disjoint subsets $S_k$  such that for each $k$, the set $\{n_i \, |\, i \in S_k\}$ is a basis for a finitely generated multiplicative group $P_k \subseteq \Q_{\geq 0}.$

Consider $\Omega$-algebras on $s$ colours with arities as above and the set of identities $\Sigma = \Sigma_1 \cup \Sigma_2$, where $\Sigma_2=\Sigma_{2_1}\cup\Sigma_{2_2}$ is given as follows:

$\Sigma_{2_1}$ is given by the following order-preserving identifications (as in the Brown-Stein algebra in Example \ref{stein}): for each $k \leq m$ we have

$$\{\alpha_i^1\alpha_j^1,\ldots,\alpha_i^1\alpha_j^{n_j},\alpha_i^2\alpha_j^1,\ldots,\alpha_i^2\alpha_j^{n_j},\ldots,\alpha_i^{n_i}\alpha_j^1,\ldots,\alpha_i^{n_i}\alpha_j^{n_j}\}=$$
$$\{\alpha_j^1\alpha_i^1,\ldots,\alpha_j^1\alpha_i^{n_i},\alpha_j^2\alpha_i^1,\ldots,\alpha_j^2\alpha_i^{n_i},\ldots,\alpha_j^{n_j}\alpha_i^1,\ldots,\alpha_j^{n_j}\alpha_i^{n_i}\},$$
where $i \neq j$ and $i,j\in S_k.$

$\Sigma_{2_2}$ is given by Brin-like identifications (as in Example \ref{brinexamples}): 
for all $i \in S_k$ and $j \in S_l$ such that $S_k \cap S_l = \emptyset$ $ (k\neq l,k,l \leq m)$, we have

$$\Sigma_{2_2}:=\{\alpha_i^l\alpha_j^t=\alpha_j^t\alpha_i^l\mid
1\leq l\leq n_i; 1\leq t\leq n_j \}.$$

We call the resulting Cantor algebra $U_r(\Sigma)$ Brin-like and denote the generalised Higman-Thompson group by $V_r(\Sigma)=V_{\{n_i \,|\,i\in S_1\},...,\{n_i \,|\,i\in S_m\}}.$
\end{definition}

\begin{example}
From Definition \ref{def-brinlike} we notice the following examples:
\begin{enumerate}
\item If $m=s$, we have the Brin-groups as in Example \ref{brinexamples} (iii).
\item If $m=1$, we have Stein-groups as in Example \ref{stein}.
\item Suppose we have that $\{n_i \, |\, i\in S_k\} = \{n_i \, |\, i\in S_l\}$ for each $l,k \leq m$. Then the resulting group can be viewed as a higher dimensional Stein-group $mV_{\{n_i \, |\, i\in S_m\}}$.
\end{enumerate}
\end{example}

\begin{question}
Suppose $m\notin \{1,s\}.$  What are the conditions on the arities for the groups $V_{\{n_i \,|\,i\in S_1\},...,\{n_i \,|\,i\in S_m\}}$  not 
be isomorphic to any of the known generalised Thompson groups such as the Higman-Thompson groups, Stein's groups or Brin's groups? More generally, when are two of these groups
non-isomorphic? See \cite{warrenconcha} for some special cases.
\end{question}

\begin{remark}\label{cantorcube}
We can view these groups as bijections of $m$-dimensional cuboids in the $m$-dimensional Cartesian product of the Cantor-set, similarly to the description given for $sV$, the Brin-Thompson groups. In each direction, we get subdivisions of the Cantor-set as in the Stein-Brown groups given by $\Sigma_{2_1}$. 
\end{remark}

\begin{lemma}
The Brin-like Cantor-algebras are valid and bounded.
\end{lemma}

\begin{proof}
Using the description given in Remark \ref{cantorcube} we can apply the same argument as in \cite[Lemma 3.2]{desiconbrita2}.
\end{proof}

\noindent The groups defined in this subsection all satisfy the following condition on the relations in $\Sigma$, and hence satisfy the conditions needed in Section \ref{fpinfty}.

\begin{definition}\label{complete}
Using the notation of Definition \ref{sigmadef}, suppose that for all $i\neq i'$, $i,i' \in S$ we have that $\Sigma_2^{i,i'} \neq \emptyset$ and that $f(j)=i'$ for all $j=1,...,n_i$ and $f'(j')=i$ for all $j'=1,...,n_{i'}$.
Then we say that $\Sigma$ (or equivalently $U_r(\Sigma)$) is {\sl complete}.
\end{definition}

\begin{remark}\label{brinlikecomplete} The Brin-like Cantor-algebras are complete.
\end{remark}

\section{Finiteness conditions}\label{fpinfty}

\noindent In this section we prove the following result:

\begin{theorem}\label{FPinfty} Let $\Sigma$ be valid, bounded and complete. Then $V_r(\Sigma)$ is of type $\F_\infty$.
\end{theorem}

We closely follow  \cite{fluch++}, where it is shown that Brin's groups $sV$ are of type $\F_\infty.$  We shall use  a different notation, which is more suited to our set-up, and will explain where the original argument has to be modified to get the more general case. Throughout this section  $U_r(\Sigma)$ denotes a  valid, bounded and complete Cantor-algebra.

\begin{definition} Let $B\leq A$ be admissible subsets of $U_r(\Sigma)$. We say that the expansion $B\leq A$ is {\sl elementary} if there are no repeated colours in the paths from leaves in $B$ to their descendants in $A$.  Since $\Sigma$ is complete, this condition is preserved by the relations in $\Sigma$. We denote an elementary expansion by $B\preceq A.$  We say that the expansion is {\sl very elementary} if all paths have length at most 1. In this case we write $B\sqsubseteq A$.
\end{definition}
 
\begin{remark} If $A\preceq B$ is elementary (very elementary) and $A\leq C\leq B$, then $A\preceq C$ and $C\preceq B$ are elementary (very elementary)).\end{remark}

\begin{lemma}\label{thm:unique-descendant}
Let $\Sigma$ be  complete, valid and bounded. Then any admissible basis $A$ has a unique maximal elementary admissible descendant, denoted by $\mathcal{E}(A).$
\end{lemma}
\begin{proof} Let $\mathcal{E}(A)$ be the admissible subset of $n_1\ldots n_s|A|$ leaves obtained by applying all descending operations exactly once to every element of $A$. \end{proof}

\subsection{The Stein subcomplex}

Denote by  $\mathcal{P}_r$ the poset of of admissible bases in $U_r(\Sigma).$ The same argument as in  \cite[Lemma 3.5 and Remark 3.7]{desiconbrita2}  shows that its geometric realisation $|\mathcal{P}_r|$ is contractible, and that $V_r(\Sigma)$ acts on $\mathcal{P}_r$ with finite stabilisers. In \cite{desiconbrita2, britaconcha} this poset was denoted by $\mathfrak{A}$, but here we will follow the notation of \cite{fluch++}. This poset is essentially the same as the poset of \cite{fluch++} denoted $\mathcal{P}_r$ there as well.

We now construct the Stein complex $\mathcal{S}_r(\Sigma)$, which is a subcomplex of $|\mathcal{P}_r|$. The vertices in $\mathcal{S}_r(\Sigma)$ are given by the admissible subsets
of $U_r(\Sigma)$. The  $k$-simplices are given by chains of  expansions $Y_0\leq\ldots\leq Y_k$, where $Y_0\preceq Y_k$ is an elementary expansion.

\begin{lemma}\label{core}
Let $A,B \in \mathcal{P}_r$ with $A< B$. There exists a unique $A< B_0\leq B$ such that $A\prec B_0$ is elementary and for any $A\leq C \leq B$ with $A\preceq C$ elementary we have $C \preceq B_0.$ 
\end{lemma}

\begin{proof}
Let $\mathcal{E}(A)$ be as in the proof of Lemma \ref{thm:unique-descendant}.
Let $B_0=\glb(\mathcal{E}(A),B)$ which exists by \cite[Lemma 3.14]{desiconbrita2}. If $A\preceq C\leq B$, then $C\leq\mathcal{E}(A)$ and so $C\leq B_0$.
\end{proof}

\begin{lemma}\label{contractible}
For every $r$ and every valid, bounded and complete $\Sigma$, the Stein-space  $\mathcal{S}_r(\Sigma)$
is contractible.
\end{lemma}

\begin{proof} By \cite[Lemma 3.5]{desiconbrita2}, $|\mathcal{P}_r|$ is contractible. Now use the same argument of \cite[Corollary 2.5]{fluch++} to deduce that $\mathcal{S}_r(\Sigma)$ is homotopy equivalent to $|\mathcal{P}_r|$. Essentially, the idea is to use Lemma \ref{core} to show that each simplex in $|\mathcal{P}_r|$ can be pushed to a simplex in $\mathcal{S}_r(\Sigma)$.
\end{proof}

\begin{remark} Notice that  the action of  $V_r(\Sigma)$ on $\mathcal{P}_r$ induces an action of $V_r(\Sigma)$ on $\mathcal{S}_r(\Sigma)$ with finite stabilisers.
\end{remark}

\noindent Consider the Morse function $t(A)=|A|$ in $\mathcal{S}_r(\Sigma)$ and filter the complex with respect to $t$, i.e.

 $$\mathcal{S}_r(\Sigma)^{\leq n}:=\{A\in\mathcal{S}_r(\Sigma)\mid t(A)\leq n\}.$$
By the same argument as in \cite[Lemma 3.7]{desiconbrita2} $\mathcal{S}_r(\Sigma)^{\leq n}$ is finite modulo the action of $V_r(\Sigma)$. 
Let $\mathcal{S}_r(\Sigma)^{< n}=\{A\in\mathcal{S}_r(\Sigma)\mid t(A) <n\}.$ 



Provided that
 \begin{equation}\label{pair} \text{ the connectivity of the pair
 } (\mathcal{S}_r(\Sigma)^{\leq n},\mathcal{S}_r(\Sigma)^{<n})\text{ tends to $\infty$ as $n\to\infty$},\end{equation}
 Brown's Theorem \cite[Corollary 3.3]{brown2} implies that $V_r(\Sigma)$ is of type $\F_\infty,$ thus proving Theorem \ref{FPinfty}. The rest of this section is devoted to proving (\ref{pair}).

\subsection{Connectivity of descending links}
Recall that for any $A\in\mathcal{S}_r(\Sigma)$ the descending link $L(A):=\downlink^t(A)$ with respect to $t$ is defined to be the intersection of the link 
$\downlink(A)$ with $\mathcal{S}_r(\Sigma)^{<n}$,
where $t(A)=n$.
To show (\ref{pair}), we proceed as in \cite{fluch++}.
Using Morse theory, the problem is reduced to showing that for $A$ as before,
 the connectivity of $L(A)$ tends to $\infty$ when $t(A)=n\to\infty$. Whenever this happens, we will say that $L(A)$ is {\sl $n$-highly connected}. More generally: assume we have a family of complexes $(X_\alpha)_{\alpha\in\Lambda}$ together with a  map $n: (X_\alpha)_{\alpha\in\Lambda} \to \Z_{>0}$ such that the set $\{n(\alpha)_{\alpha\in\Lambda}\}$ is unbounded. Assume further that whenever $n(\alpha)\to\infty$,  the connectivity of the associated $X_\alpha$s tends to $\infty$. In this case we will say that the family is {$n$-highly connected}. 
 
Note that
$L(A)$ is the subcomplex of $\mathcal{S}_r(\Sigma)$ generated by $$\{B\mid B\prec A\text{ is an elementary expansion}\}.$$
Following \cite{fluch++}, define a height function $h$ for $B\in L(A)$ as follows: 
$$h(B):=(c_s,\ldots,c_2,b)$$
 where $b=|B|$ and $c_i$ ($i=2,\ldots,s$) is the number of leaves in $A$ whose length as descendants of their parent in $B$ is $i$. We order these heights lexicographically. Let $c(B)=(c_s,\ldots,c_2),$ which are also ordered lexicographically.
Denote by
$L_0(A)$ the subcomplex of $\mathcal{S}_r(\Sigma)$ generated by $\{B\mid B\sqsubset A\text{ is a very elementary expansion}\}.$
Then for any $B\in L(A)$, $B\in L_0(A)$ if and only if $h(B)=(0,\ldots,0,|B|)$. 

\begin{lemma}\label{first} The set of complexes of the form $L_0(A)$ is $t(A)$-highly connected.
\end{lemma}
\begin{proof} For any $n\geq 0$, we define a complex denoted $K_n$ as follows. Start with a set $A$ with $n$ elements. 
The vertex set of $K_n$ consists of labelled subsets of $A$ where the possible labels are the colours $\{1,\ldots,s\},$ and where a subset labelled $i$ has precisely $n_i$ elements. Recall that $n_i$ is the arity of the colour $i$.
A $k$-simplex $\{\sigma_0,\ldots,\sigma_k\}$ in $K_n$ is given by  an unordered set of pairwise disjoint $\sigma_j$s. This complex is isomorphic to the barycentric subdivision of $L_0(A)$ for $n=t(A)$.
To prove that $K_n$ is $n$-highly connected,
proceed as in the proof of \cite[Lemma 4.20]{brown2}.
\end{proof}

Now consider descending links in $L(A)$  with respect to the height function $h$, i.e. for $B\in L(A)$ let 
$\downlink^h(B)$ be the subcomplex of $L(A)$ generated by $\{C\in L(A)\mid h(C)\leq
h(B) \mbox{ and either } B<C \mbox{ or } C>B\}.$
Consider the following two cases:
\begin{itemize}
\item[i)] $B\in L(A)\setminus L_0(A)$ and there is at least one leaf of $B$ that is expanded precisely once to obtain $A$.

\item[ii)] $B\in L(A)\setminus L_0(A)$ and no leaf of $B$ is expanded precisely once to obtain $A$.
\end{itemize}  

\noindent The next two Lemmas show that in either case $\downlink^h(B)$ is $t(A)$-highly connected. 

As in \cite{fluch++} the descending link $\downlink^h(B)$ of some $B\in L(A)$ with respect to $h$ can be viewed as the join of two subcomplexes, the down-link  and the up-link. The downlink consists of those elements $C$ such that $C < B$ and $h(C) \leq h(B)$. Hence, by the above $c(B)=c(C).$ The uplink consists of those $C$ that  $B< C$, $h(C) \leq h(B)$, and therefore $c(B) > c(C).$  

\begin{lemma}\label{second}
Let $B \in L(A)$ as in i). Then $\downlink^h(B)$ is contractible.
\end{lemma}
\begin{proof} It suffices to follow the proof of \cite[Lemma 3.7]{fluch++}. We briefly sketch this proof using our notation: let $b\in B$ be a leaf that is expanded precisely once to obtain $A$. 
Let $B\prec M\preceq A$ be the maximal elementary expansion of $B$ that preserves the leaf $b$ and lies over $A$. The existence of $M$ follows from a variation of Lemma \ref{core}. Now, for any $C\in\downlink^h(B)$ lying in the uplink we let $B\prec C_0\sqsubseteq C$, where $C_0$ is obtained by  performing all expansions in $B$ needed to get $C$, except the one of $b$. 

One easily checks that $C_0\leq M$, that $C_0$ and $M$ lie in $\downlink^h(B)$ and that both $C_0$
and $M$ lie in the uplink.  Hence $M\geq C_0\leq C$ provides a contraction of the uplink. 
As $\downlink^h(B)$ is the join of the downlink and the uplink we get the result.
\end{proof}

\begin{lemma}\label{third} Let $B$ be as in ii). Then $\downlink^h(B)$ is $t(A)$-highly connected.
\end{lemma}

\begin{proof} As before, we follow the proof of \cite[Lemma 3.8]{fluch++} with only minor changes. With our notation, we let $k_s$ be the number of leaves on $B$ that are also leaves of $A$ and let $k_b$ be the remaining leaves. Then one checks that the the up-link in $\downlink^h(B)$ is $k_b$-highly connected and that the down-link is $k_s$-highly connected. As $t(A)=n\leq k_bn_1\ldots n_s+k_s$, we get the result.
\end{proof}

\noindent Finally, using Morse theory as in \cite{fluch++}, we deduce that the pair $(L(A),L_0(A))$ is $t(A)$-highly connected. As a result, $L(A)$ is also $t(A)$-highly connected, establishing (\ref{pair}) and hence Theorem \ref{FPinfty}.

\medskip\noindent Some time after a preprint of this work was posted, we learned of Thumann's work \cite{thumann1, thumann2}, where he provides a generalised framework of groups defined by operads to apply the techniques introduced in \cite{fluch++}.  We believe that automorphism groups of valid, bounded and complete Cantor algebras might be obtained making a suitable choice of cube cutting operads, see \cite[Subsection 4.2]{thumann1}. Therefore Theorem 4.1 could also be seen as a special case of \cite[Subsection 10.2]{thumann2}.

\section{Finiteness conditions for centralisers of finite subgroups}\label{centralisersection}

\noindent From now on, unless mentioned otherwise, we assume that the Cantor-algebra $U_r(\Sigma)$ is valid and bounded.

\begin{definition} Let $L$ be a finite group. The set of bases in $U_r(\Sigma)$ together with the expansion maps can be viewed as a directed graph. Let $(U_r(\Sigma),L)$ be the following diagram of groups associated to this graph: To each basis $A$ we associate $\text{Maps}(A,L)$,
the set of all maps from $A$ to $L$.
Each simple expansion $A\leq B$ corresponds to the diagonal map $\delta:\text{Maps}(A,L)\to\text{Maps}(B,L)$ with $\delta(f)(a\alpha_i^j)=f(a)$, where $a\in A$ is the expanded leaf, i.e.  
$B=(A\setminus\{a\})\cup\{a\alpha_i^1,\ldots, a\alpha_i^{n_i}\}$ 
for some colour $i$ of arity $n_i$. 
To arbitrary expansions we associate the composition of the corresponding diagonal maps.

\end{definition}

Centralisers of finite subgroups in $V_r(\Sigma)$ have been described in \cite[Theorem 4.4]{britaconcha} 
and also in \cite[Theorem 1.1]{matuccietc} for the Higman-Thompson groups $V_{n,r}$. This last description is more explicit and makes use of the action of $V_{n,r}$ on the Cantor set (see Remark \ref{topology} below).
 
We will use the following notation, which was  used in \cite{britaconcha}:
let $Q\leq V_r(\Sigma)$ be a finite subgroup and let $t$ be the number of transitive permutation representations $\varphi_i: Q\to S_{m_i}$ of $Q$. Here, $1\leq i\leq t$, $m_i$ is the orbit length and $S_{m_i}$ is the symmetric group of degree $m_i$. Also let  $L_i=C_{S_{m_i}}(\varphi_i(Q))$. 

\noindent There is a basis $Y$ setwise fixed by $Q$ and which is of  minimal cardinality. The group $Q$ acts on $Y$ by permutations.  Thus there exist integers 
$0\leq r_1,\ldots,r_t\leq d$ 
such that 
$Y=\bigcup_{i=1}^t W_i$ with $W_i$  the union of exactly $r_i$ $Q$-orbits 
of type $\varphi_i$. See Remark \ref{adm-basis} for the definition of $d$.

The next result combines the descriptions in \cite[Theorem 4.4]{britaconcha}
and \cite[Theorem 1.1]{matuccietc} giving a more detailed description of the centralisers of finite subgroups in $V_r(\Sigma)$.

\begin{theorem}\label{centraliser} Let $Q$ be a finite subgroup of $V_r(\Sigma).$ Then
$$C_{V_r(\Sigma)}(Q)=\prod_{i=1}^tG_i$$ where  
$G_i=K_i\rtimes V_{r_i}(\Sigma)$ and $K_i=\colim(U_{r_i}(\Sigma),L_i)$. Here, $V_r(\Sigma)$ acts on $K_i$ as follows: let $g\in V_{r_i}(\Sigma)$ and let $A$ be a basis in $U_{r_i}(\Sigma)$. The action of $g$ on $K_i$ is induced, in the colimit, by the map $\text{Maps}(A,L)\to\text{Maps}(gA,L)$ obtained contravariantly from $gA\buildrel g^{-1}\over\to A.$ 
\end{theorem}

\begin{proof} The decomposition of $C_{V_r(\Sigma)}(Q)$ into a finite direct product of semidirect products was shown in \cite[Theorem 4.4]{britaconcha}.  Hence, for the first claim, all that remains to be  checked is that $K_i=\colim (U_{r_i}(\Sigma),L_i)$. We use the same notation as in the proof of \cite[Theorem 4.4]{britaconcha}.

 Fix $\varphi=\varphi_i$, $l:=r_i$, $L:=L_i$, $m:=m_i$ and $K:=K_i=\ker \tau.$   Let $x\in K= \ker \tau$, where $\tau:  C_{V_r(\Sigma)}(Q) \epi V_l(\Sigma)$ is the split surjection of the proof of \cite[Theorem 4.4]{britaconcha}.  With $Y$ as above,  there is a basis $Y_1\geq Y$ with $xY_1=Y_1$ and $Y_1$ is also $Q$-invariant. Then the basis $Y_1$ decomposes as a union of $l$ $Q$-orbits (all of them of type $\varphi$), and $x$ fixes these orbits setwise. We denote these orbits by $\{C_1,\ldots,C_l\}$. In each of the $C_j$ 
 there is a marked element. Since $\varphi$ is transitive this can be used to fix a bijection $C_j\to\{1,\ldots,m\}$ corresponding to $\varphi$. Then the action of $x$ on $C_j$ yields a well defined $l_j\in L$. This means that we may represent $x$ as $(l_j)_{1\leq j \leq l}$. Let $A$ be the basis of $U_l(\Sigma)$ obtained from $Y_1$ by identifying all elements in the same $Q$-orbit, i.e. $A=\tau^{\frak U}(Y_1)$ with the notation of \cite{britaconcha}. Denote $A=\{a_1,\ldots,a_l\}$ with $a_j$ coming from $C_j$. Then the element $x$ described before can be viewed as the map $x:A\to L$ with $x(a_j)=l_j$. Suppose we chose a different basis $Y_2$ fixed by $x$. It is a straightforward check to see that there is a basis $Y_3$ also fixed by $x$, such that  $Y_1,Y_2\leq Y_3$, and that this representation is compatible with the associated expansion maps.

To prove the second claim, consider an element $g\in V_l(\Sigma)$ viewed as an element in  $C_{V_r(\Sigma)}(Q)$ using the splitting $\tau$ above. This means that $g$ maps $Q$-fixed bases to $Q$-fixed bases 
and that $g$ preserves the set of marked elements.
Let $ Y_1,A $ and $x\in K$ be as above. Then the basis $gY_1$ is the union of the $Q$-orbits $\{gC_1,\ldots,gC_l\}$ and $\tau^{\frak U}(gY_1)=gA.$ Also, for any $c_i\in C_i$, $gxg^{-1}gc_i=gxc_i$ which means that if the action of $x$ on $C_i$ is given by $l_i\in L$, then the action of $x^g$ on $gC_i$ is given also by $l_i$. Therefore the map $gA\to L$ which represents $x^g$ is the composition of the maps $g^{-1}:gA\to A$ and the map $A\to L$ which represents $x$.
\end{proof}

\begin{remark}\label{topology} In \cite{matuccietc}, where the ordinary Higman-Thompson group $V_r(\Sigma)=V_{n,r}$ is considered, the subgroups $K_i$ are described as $\text{Map}^0(\mathfrak{C},L)$,
where $\mathfrak{C}$ denotes the Cantor set, and $\text{Map}^0$ the set of continuous maps. 
Here the Cantor set is viewed as the set of right infinite words in the descending operations.

It is a straightforward  check to see that both descriptions are equivalent in this case. In fact $x:A\to L$ corresponds to the element in $\text{Map}^0(\mathfrak{C},L)$ mapping each $\varsigma\in\mathfrak{C}$ to $x(a)$ for the only $a\in A$ which is a prefix of $\varsigma$. 
Similarly, one can describe $K_i$ when  $V_{r_i}(\Sigma)=sV$ is a Brin-group, using the fact that these groups act on $\mathfrak{C}^s$, see \cite{warrenconcha}.
\end{remark}

\begin{notation}  Let
$$\Omega:=\{B(\L)\mid B\subset\mathcal{L}\text{ finite}\}\cup\{\emptyset\}.$$

We also denote
$$\Omega^n:=\Omega\times\buildrel n\over\ldots\times\Omega=\{(\omega_1,\ldots,\omega_n)\mid\omega_i\in\Omega\},$$
$$\Omega^n_{c}:=\{(\omega_1,\ldots,\omega_n)\in\Omega^n\mid\cup_{i=1}^n\omega_i=\L\}.$$
\end{notation}

\begin{lemma}\label{properties}
\begin{itemize}
\item[i)]  Let $B \geq A \geq X$ be bases and $B_1\subseteq B$. Let $A_1:=\{a\in A\mid a\text{ is a prefix of an element in }B_1\}$. Then $A_1(\L)=B_1(\L)$.

\item[ii)] Let $A\geq X$ be a basis, then $A(\L)=\L$.

\item[iii)] For any $(\omega_1,\ldots,\omega_n)\in\Omega^n$ there is some basis $A$ with $X\leq A$ and some $A_i\subseteq A$, $1\leq i\leq n$ such that $\omega_i=A_i(\L)$.

\item[iv)] Let $A\geq X$ be a basis, $A_1,A_2\subseteq A$ and $\omega_i=A_i(\L)$ for $i=1,2$. Then $\omega_1=\omega_2$ if and only if $A_1=A_2$.

\item[v)] Let $A,B \geq X$ be two bases and $\omega\in\Omega$ be such that for some $A_1\subseteq A$, $B_1\subseteq B$ we have $\omega=A_1(\L)=B_1(\L)$. Then $|A_1|\equiv |B_1|$ mod $d$ and $|A_1|=0$ if and only if $|B_1|=0.$

\item[vi)] Let $A,B \geq X$ be two bases and $A_1,A_2\subseteq A$, $B_1,B_2\subseteq B$ with $A_1(\L)=B_1(\L)$ and $A_2(\L)=B_2(\L)$. Then $A_1\cap A_2=\emptyset$ if and only if $B_1\cap B_2=\emptyset$.
\end{itemize}
\end{lemma}
\begin{proof} It suffices to prove i) in the case when $B$ is obtained by a simple expansion from $A$. Moreover, we may assume that $A_1=\{a\}$ and $B_1=\{a\alpha_i^1,\ldots,a\alpha_i^{n_i}\}$ for some colour $i$ of arity $n_i$. Then obviously $B_1(\L)\subseteq a(\L).$ Denote $b_j=a\alpha_i^j$ and let $u\in a(\L)$. Then $uv=ac$ for descending words $v$ and $c$. Performing the descending operations given by $c$ on the basis $A$, we obtain a  basis $C$ with $ac\in C$. Let $D$ be a basis with $C,B\leq D$. Then there is some element $d\in D$ which can be written as $d=acc'$ for some descending word $c'$. Moreover, Remark \ref{tec} also implies that $d=b_jb'$ for some $j$ and descending word $b'$. As $uvc'=acc'=b_jb'$ we get $u\in b_j(\L)$.

\noindent  Now ii) follows from i).

\noindent To prove iii), suppose that  $\omega_i=\{a_i^1,\ldots,a_i^{l_i}\}(\L)$.  For each $a_i^j$ we may find a basis $T_i^j\geq X$ containing $a_i^j$. Now let $A$ be common descendant of the $T_i^j$ and use  i).

\noindent  To establish  iv), it suffices to check that if $\widehat a\in A,$ $\widehat a\not\in A_i$, then 
$\widehat a\not\in A_i(\L)$. Suppose $\widehat a\in A_i(\L).$  Then there are  descending words $v,u$ and some $a\in A_i$, such that $\widehat av=au=b$. Performing the descending operations given by $v$ and $u$ on $\widehat a$ and $a$ respectively, we get a basis $A\leq B$ and $b \in B$  contradicting Remark \ref{tec}.

\noindent In v), since there is a basis $C$ with $A,B\leq C$, we may assume $A\leq B.$ Then  v) is a consequence of i) and iv).

\noindent Finally, for vi) we may also assume $A\leq B$. Then we only have to use Remark \ref{tec}.
\end{proof}

\begin{notation} Let $\omega\in\Omega$, $X\leq A$ and $B\subseteq A$ such that $\omega=B(\L)$.
We  put
$$\|\omega\|=\Bigg\{\begin{aligned}
&0\text{ if }\omega=\emptyset\\
&t \text{ for }|B|\equiv t\text{ mod }d\text{ and }0<t\leq d\text{ otherwise.}\\
\end{aligned}$$
This is well defined by Lemma \ref{properties} v).
Take $B'\subseteq A$ and $\omega'=B'(\L)$. If $B\cap B'=\emptyset$, we put $\omega\wedge\omega'=\emptyset$. Note that by Lemma \ref{properties} vi) this  is well defined.

Finally, let
$$\Omega^n_{c,\text{dis}}:=\{(\omega_1,\ldots,\omega_n)\in\Omega^n_c\mid \L=\bigcup_{i=1}^n\omega_i\text{ and }\omega_i\wedge\omega_j=\emptyset\text{ for }i\neq j
\}.$$

\end{notation}

The group $V_r(\Sigma)$ does not act on the set of leaves. It does, however, act on $\Omega$ as we will see in Lemma \ref{tec1}. Nevertheless, if $l$ is a leaf such that $l\in A$ for a certain basis $A \geq X$ and $g$ is a group element such that $gA \geq X$, then
we will denote by $gl$ the leaf of $gA$ to which $l$ is mapped by $g$.

\begin{lemma}\label{tec1} The group $V_r(\Sigma)$ acts by permutations on $\Omega$ and on $\Omega^n_{c,\text{dis}}$.  There are only finitely many $V_r(\Sigma)$-orbits under the latter action. Furthermore, the stabiliser of any element in $\Omega^n_{c,\text{dis}}$ is of the form
$$V_{k_1}(\Sigma)\times\ldots\times V_{k_n}(\Sigma)$$
for certain integers $k_1,\ldots,k_n$.
\end{lemma}

\begin{proof} To see that $V_r(\Sigma)$ acts on $\Omega$, it suffices to check that if $\omega=l(\L)$ for some leaf $l\in\L$, we have $g\omega\in\Omega$ for any $g\in V_r(\Sigma)$. Let $X\leq A$ be a basis with $l\in A$. By Remark \ref{tecrem1} there is some $A\leq B$ with $A\leq gB$. Note that by Lemma \ref{properties} i) $\omega$ can also be written as
$$\omega=B_1(\L)$$
where $B_1=\{l_1,\ldots,l_k\}$ is the set of leaves in $B$ obtained from $l$. Therefore $gB_1=\{gl_1,\ldots,gl_k\}\subseteq gB$ and $g\omega =gB_1(\L).$

\noindent That this action induces an action on $\Omega^n_{c,\text{dis}}$ is a consequence of the easy fact that for any $g\in V_r(\Sigma)$ and any $(\omega_1,\ldots,\omega_n)\in\Omega^n_{c,\text{dis}}$ we have $g\omega_i\wedge g\omega_j=\emptyset$ and $\L=\cup_{i=1}^ng\omega_i$.

Let $(\omega_1,\ldots,\omega_n),(\omega_1',\ldots,\omega_n')\in\Omega^n_{c,\text{dis}}$ be
such that
$\|\omega_i\|=\|\omega_i'\|$ for $1\leq i\leq n$. There are bases $X\leq A,A'$ and subsets $A_1,\ldots,A_n\subseteq A$, $A'_1,\ldots,A'_n\subseteq A'$ such that for each $1\leq i\leq n$, $\omega_i=A_i(\L)$, $\omega_i'=A_i'(\L)$ and $|A_i|=|A_i'|$. Hence we may choose a suitable element $g\in V_r(\Sigma)$ such that $gA=A'$ and $gA_i=A_i'$
for each $i=1,\ldots,n$. Then $g(\omega_1,\ldots,\omega_n)=(\omega_1',\ldots,\omega_n')$. Since the number of possible $n$-tuples of integers modulo $d$ having the same number of zeros is finite, it follows that there are only finitely many  $V_r(\Sigma)$-orbits.

Finally consider $\mathcal{W}=(\omega_1,\ldots,\omega_n)\in\Omega^n_{c,\text{dis}}$ 
as before, i.e. with $X\leq A$ and $A_1,\ldots, A_n\subseteq A$ such that $\omega_i=A_i(\L)$ for $1\leq i\leq n$.
An element $g\in V_r(\Sigma)$ fixes $\mathcal{W}$ if and 
only if $g\omega_i=\omega_i$ for each $i=1,\ldots,n$. 
We may choose a basis $B$ with $A\leq B,gB$ and then, by using Lemma \ref{properties} i) and iv), we see that $g$  fixes $\mathcal{W}$ if and only if it maps those leaves of $B,$ which are of the form $av$ for some $a\in A_i$ and some descending word $v$, to the analogous subset in $gB$. Considering the subalgebra of $U_r(\Sigma)$ generated by the $A_i$, we see that $g$ can be decomposed as $g=g_1\ldots g_n$ with $g_i\in V_{k_i}(\Sigma)$ for $k_i=|A_i|$.
\end{proof}

Let $K$ be a group and denote by $Y=K\ast K\ast\ldots$  the
infinite join of copies of $K$ viewed as a discrete $CW$-complex, i.e. $Y$ is the space
obtained by Milnor's construction for $K$. 
Then $Y$ has a $CW$-complex decomposition
whose associated chain complex yields the standard bar resolution. For detail see, for example, \cite[Section 2.4]{bensonII}.

Obviously, if a group $H$ acts on $K$ by conjugation, this action can be extended to an action of $H$ on $Y$ and to an
action of $G=K\rtimes H$ on $Y$.

\begin{lemma}\label{auxmainfinfty}
Let $H$ and $K$ be groups and let $H$ act on $K$ via $\varphi: H \to \mathrm{Aut} K$.
Assume that $H$ is of type $\F_{\infty}$, and that for every $n \in \mathbb{N}$ the induced action of
$H$ on $K^n$ has finitely many orbits and has stabilisers of type $\F_{\infty}$. Then $G = K \rtimes_{\varphi} H$ is of
type $\F_{\infty}$.
\\ \noindent The same statement holds if $\F_\infty$ is replaced with $\FP_\infty$.
\end{lemma}

\begin{proof}
Let $Y_n = K^{\ast n}$ and let $Y$ be as above.
Consider the action of $G$ on $Y$ induced by the diagonal action. 
Note that this preserves the individual join factors.
Since the action of $K$ on $Y$ is free, he stabiliser of a cell in $G$ is isomorphic to its
stabiliser in $H$. The stabiliser of an $(n-1)$-simplex is the stabiliser of $n$  elements of $K$, thus $\F_{\infty}$ by assumption. Maximal simplices in $Y_n$
correspond to elements of $K^n$ and every simplex of $Y_n$ is contained in a maximal simplex. This, together with the fact that the action of $G$ on $K^n$ has only finitely many orbits, implies that the action of $G$ on $Y_n$ is
cocompact. Finally, the connectivity of the filtration $\{Y_n \}_{n \in 
\mathbb{N}}$ tends to infinity as $n \to \infty.$
Hence the claim follows from \cite[Corollary 3.3(a)]{brown2}.
\end{proof}

\begin{theorem}\label{mainfinfty} Assume that for any $t>0$, the group $V_t(\Sigma)$ is of type $\F_\infty.$ Then the
groups  $G_i= K_i\rtimes V_{r_i}(\Sigma)$ of Theorem \ref{centraliser} are of type $\F_\infty.$ 

The same statement holds if $\F_\infty$ is replaced with $\FP_\infty$.
\end{theorem}

\begin{proof}  Put $V:=V_r(\Sigma)$, $K:=K_i$ and $G:=G_i$.
We claim that for every $n$ there is some $\overline n$ big enough such that there is an
injective map of $V$-sets
$$\phi_{n}:K^{n}\to\Omega^{\overline n}_{c,\text{dis}}.$$
 Let $x\in K$ be given by a map $x: A \to L$, where $A$ is a basis with $X \leq A.$ The element $x$ is determined uniquely by a map which, by slightly abusing notation,  we also denote 
$x:L\to\Omega$. This $x$ maps any $s\in L$ to $\omega_s:=A_s(\L)$ with $A_s=\{a\in A\mid x(a)=s\}.$ Obviously $\cup_{s\in
L}\omega_s=\L$. This means that fixing an order in $L$ yields an injective
map of $V$-sets
$$\xi_n:K^n\to\Omega^{n|L|}_c.$$
Consider any $(\omega_1,\ldots,\omega_m)\in\Omega^m_c$ for $m=n|L|$. Let $X\leq A$ with $A_1\ldots, A_m\subseteq A$ and $\omega_i=A_i(\L)$ for $1\leq i\leq m$. Let  $\overline
n:=2^m-1$, i.e. the number of non-empty subsets $\emptyset\neq
S\subseteq\{1,\ldots,m\}$. For any such $S$ let
$$A_S:=\bigcap_{i\in S}A_i\setminus\cup\{\bigcap_{j\in T}A_j\mid {S\subset
T\subseteq\{1,\ldots,m\}}\}.$$
Then one easily checks that the $A_S$ are pairwise disjoint and that their union is $\L$. Let $\omega_S:=A_S(\L)$. The preceding paragraph means that fixing an ordering on the set of non-empty
subsets of $\{1,\ldots,m\}$ yields an injective map of $V$-sets
$$\rho_m:\Omega^m_c\to\Omega^{\overline n}_{c,\text{dis}}.$$
Composing $\xi_n$ and $\rho_{m}$ we get the desired $\phi_n$.

Now, applying Lemma \ref{tec1} we deduce that $K^n$ has only finitely many orbits under
the action of $V_r(\Sigma)$ and that every cell stabiliser is isomorphic to a direct
product of copies of $V_{t}(\Sigma)$ for suitable indices $t$. It now suffices to use Lemma \ref{auxmainfinfty}.
\end{proof}

\goodbreak

This implies that \cite[Conjecture 7.5]{britaconcha} holds.

\begin{corollary}\hfill
\begin{enumerate}
\item $V_r(\Sigma)$ is quasi-$\underline{\FP}_\infty$ if and only if $V_k(\Sigma)$ is of
type ${\FP}_\infty$ for any $k.$
\item $V_r(\Sigma)$ is quasi-$\underline{\FFF}_\infty$ if and only if $V_k(\Sigma)$ is of
type ${\FFF}_\infty$ for any $k.$
\end{enumerate}
\end{corollary}
\begin{proof}  The \lq\lq only if" part of both items is proven in \cite[Remark 7.6]{britaconcha}.
 The \lq\lq if" part is a consequence of 
 \cite[Definition 6.3, Proposition 6.10]{britaconcha} and Theorem \ref{mainfinfty} above.
 \end{proof}
 
\noindent  Theorem \ref{mainfinfty} also implies that the Brin-like groups of Section \ref{fpinfty} are of type quasi-$\F_\infty:$

\begin{corollary}
Suppose $U_r(\Sigma)$ is valid, bounded and complete. Then $V_r(\Sigma)$ is of type quasi-$\F_\infty$. 

In particular, centralisers of finite groups are of type $\F_\infty$.
\end{corollary}

\section{normalisers of finite subgroups}\label{normaliser}

\noindent Let $Y$ be any basis. We denote $$S(Y):=\{g\in V_r(\Sigma)\mid gY=Y\}.$$
Observe that this is a finite group,  isomorphic to the symmetric group of degree $|Y|$.

\begin{theorem}\label{normaliser1}
Let $Q\leq V_r(\Sigma)$ be a finite subgroup. Let $Y,$ $t$, $r_i$, $l_i$, $\varphi_i$, and $1\leq i\leq t$ be as in the proof of Theorem \ref{centraliser}. Then
$$N_{V_r(\Sigma)}(Q)=C_{V_r(\Sigma)}(Q)N_{S(Y)}(Q)$$
and $N_{V_r(\Sigma)}(Q)/C_{V_r(\Sigma)}(Q)\cong N_{S(Y)}(Q)/C_{S(Y)}(Q).$
\end{theorem}
\begin{proof} Let $g\in N_{V_r(\Sigma)}(Q)$ and $Y_1=gY$. Then for any $q\in Q,$ $qY_1=qgY=gq^gY=gY=Y_1$. Therefore $Y_1$ is also fixed setwise  by $Q$.  Let $r_i'$ denote the number of components of type $\varphi_i$ in $Y_1$. Then, by \cite[Proposition 4.2]{britaconcha} $r_i\equiv r_i'$ mod $d,$ and  $r_i=0$ if and only if $r_i'=0$. 

 We claim that $Y$ and $Y_1$ are isomorphic as $Q$-sets, in other words, that $r_i=r_i'$ for every $1\leq i\leq t$. Note that since $g$ normalises $Q$, it acts on the set of $Q$-permutation representations $\{\varphi_1,\ldots,\varphi_t\}$, via $\varphi_i^g(x):=\varphi_i(x^{g^{-1}})$.
Let $i$ with $r_i\neq 0$ and let $g(i)$ be the index such that $\varphi_i^g=\varphi_{g(i)}$. The fact that $g:Y\to Y_1$ is a bijection implies that $r_i=r_{g(i)}'$. We may do the same for $g(i)$ and get an index $g^2(i)$ with $r_{g(i)}=r_{g^2(i)}'.$ At some point, since the orbits of $g$ acting on the sets of permutation representations are finite, we get $g^k(i)=i$ and $r_{g^{k-1}(i)}=r_i'.$ As $r_i'\equiv r_i$ mod $d$ we have 
$r_{g^{k-1}(i)}\equiv r_i$ mod $d,$ and since $0<r_i,r_{g^{k-1}(i)}\leq d$ we deduce that $r_i'=r_{g^{k-1}(i)}=r_i$ as claimed.

Now, we can choose an $s\in V_r(\Sigma)$ mapping $Y_1$ to $Y$
and such that $s:Y_1\to Y$ is a $Q$-map, i.e., commutes with the $Q$-action. Therefore, $s\in C_{V_r(\Sigma)}(Q)$ and $sgY=Y$ thus $sg\in N_{S(Y)}(Q).$
\end{proof}

\begin{remark} We can give a more detailed description of the conjugacy action of $N_{S(Y)}(Q)$ on the group $C_{V_r(\Sigma)}(Q)$. Recall that, by Theorem \ref{centraliser} this last group is a direct product of groups $G_1,\ldots,G_t$. We use the same notation as in Theorem \ref{centraliser}. Let $g\in N_{S(Y)}(Q)$ and put $\varphi_{g(i)}=\varphi_i^g$ as before.  Denote by $Z_{g(i)},Z_i\subseteq Y$ the subsets of $Y$ which are unions of $Q$-orbits of types $\varphi_{g(i)}$ and $\varphi_i$ respectively. Then one easily checks that $gZ_{g(i)}=Z_i$ and $G_{g(i)}=G_i^g$. Moreover, recall that $G_{i}=K_i\rtimes V_{r_i}(\Sigma)$ with $K_i=\colim(U_{r_i}(\Sigma),L_i)$ and $L_i=C_{S_{l_i}}(\varphi_i(Q))$. Then $r_{g(i)}=r_i$ and $g$ maps the subgroup $V_{r_i}(\Sigma)$ of $G_i$ to the same subgroup of $G_{g(i)}$ and $K_i$ to $K_{g(i)}$. We also notice
that $g$ acts diagonally on the system $(U_{r_i}(\Sigma),L_i)$ mapping it to $(U_{r_{g(i)}}(\Sigma),L_{g(i)})$ In particular, the action of $g$ on $L_i$ is the restriction of its action on $C_{S(V)}(Q)$ and this action yields taking to the colimit the conjugation action $K_i^g=K_{g(i)}$.
\end{remark}

\begin{remark} Using  \cite[Theorem 5]{zassenhaus}, one can also give a more detailed description of the groups $L_i$ above:
$$L_i=N_{\varphi_i(Q)}(\varphi_i(Q)_1)/\varphi_i(Q)_1$$
where $\varphi_i(Q)_1$ is the stabiliser of one letter in $\varphi_i(Q)$. Of course, if $Q$ is cyclic, then so is $\varphi_i(Q)$ and we get $\varphi_i(Q)_1=1$ and $L_i=\varphi_i(Q)$.
\end{remark}

\end{document}